 \documentclass[12pt,reqno]{amsart}

\usepackage{amssymb}

\usepackage{color}

\usepackage{amssymb}
\usepackage{amsmath}
\usepackage{amsthm}
\usepackage{times}
\usepackage{graphicx}

\usepackage{color}
\usepackage{amsthm}
\newtheorem{theorem}{Theorem}
\newtheorem{lemma}{Lemma}[section]
\newtheorem{corollary}[lemma]{Corollary}
\newtheorem{proposition}[lemma]{Proposition}

\usepackage[curve,all]{xy}
 \xyoption{curve}
 \xyoption{color}
 \xyoption{line}
\xyoption{arc}

\theoremstyle{definition}
\newtheorem{definition}[lemma]{Definition}
\newtheorem{example}[lemma]{Example}

\theoremstyle{remark}
\newtheorem{remark}[lemma]{Remark}

\numberwithin{equation}{section}

\newcommand\dashmapsto{\mapstochar\dashrightarrow}

\newcommand{\calC}{\mathcal{C}}

\newcommand{\frakS}{\mathfrak{S}}

\newcommand{\bbA}{\mathbb{A}}

\newcommand{\bbP}{\mathbb{P}}

\newcommand{\p}{\mathbb{P}}
\newcommand{\A}{\mathbb{A}}
\newcommand{\bbQ}{\mathbb{Q}}

\newcommand{\bbZ}{\mathbb{Z}}

\renewcommand{\k}{\mathrm{k}}

\newcommand{\Aut}{\textup{Aut}}

\newcommand{\Spec}{\textup{Spec}~}

\renewcommand\emptyset\varnothing

\newcommand{\beq}{\begin{equation}}
\newcommand{\eeq}{\end{equation}}
\definecolor{grey}{rgb}{0.75,0.75,0.75}
\begin{document}
\title[Automorphisms of cluster algebras of rank 2]{Automorphisms of cluster algebras of rank 2}
\thanks{The first gratefully acknowledge support by the Swiss National Science Foundation Grant  "Birational Geometry" PP00P2\_128422 /1. }


\author{J\'er\'emy Blanc}
\address{J\'er\'emy Blanc, Universit\"{a}t Basel, Mathematisches Institut, Rheinsprung $21$, $4051$ Basel, Switzerland.}
\email{jeremy.blanc@unibas.ch}

\author{Igor Dolgachev}
\address{Igor Dolgachev, Department of Mathematics, University of Michigan, 525 E. University Av., Ann Arbor, Mi, 49109, USA}
\email{idolga@umich.edu}

\dedicatory{To the memory of Andrei Zelevinsky (1953-2013)}

\begin{abstract} 
We compute the automorphism group of the affine surfaces with the coordinate ring isomorphic to a cluster algebra of rank $2$.
\end{abstract}

\maketitle
%\tableofcontents

\section{Introduction} 
\subsection{Definition of $\calC(a,b)$}
A cluster algebra $\calC(a,b)$ of rank 2 is a subring of the field of rational functions $\bbQ(y_1,y_2)$ generated by elements $y_n, n\in \bbZ,$ defined inductively by the relations
\begin{eqnarray}\label{fr}
y_{n-1}y_{n+1} = \begin{cases}
      y_n^a+1& \text{if}\  n \ \text{is even}, \\
      y_n^b+1& \text{otherwise}
\end{cases}
\end{eqnarray}
(see \cite{Sherman}). Here $a,b$ are fixed positive integers. The elements $y_n$ are called \emph{cluster variables} and the pairs $y_n,y_{n+1}$ are called \emph{clusters}. It follows from \cite{BFZ}, Corollary 1.21 that any four consecutive cluster variables, say $y_1,y_2,y_3,y_4$, generate $\calC(a,b)$ as a $\bbZ$-algebra and the relations are defining relations. Thus 
$$\calC(a,b) \cong \bbZ[y_1,y_2,y_3,y_4]/(y_1y_3-y_2^a\ -1,y_2y_4-y_3^b-1).$$

In the general context of cluster algebras, the algebra $\calC(a,b)$ corresponds to a skew-symmetrizable seed matrix 
$$\begin{pmatrix}0&a\\
-b&0\end{pmatrix}.$$
If $a = b$, it can be also defined by the quiver with two vertices and $a$ arrows from the first one to the second one. 
Note that one can also consider the case $a = b = 0$ when the algebra $\calC(0,0)$ is the algebra of Laurent polynomials in two variables. We omit this well-known case.

  The cluster algebra $\calC(a,b)$ is called of \emph{finite type} if the number of cluster variables is finite. This happens if and only if 
 $ab \le 3$ (this follows from \cite{FZ}, Theorem 6.1). In this case, we have the periodicity
 $$y_n = y_m \Leftrightarrow n\equiv m \mod h+2,$$
 where $h = 3, 4, 6$ if $ab = 1, 2, 3$, respectively. The algebra $\calC(a,b)$ is called of type $A_2,B_2,G_2$ if $ab= 1,2,3$, respectively (the remaining type $A_1\times A_1$ of a finite root system of rank 2 corresponds to the case $(a,b) = (0,0)$).
 In this paper we compute the group of automorphisms of the cluster algebra $\calC(a,b)$.
\subsection{The cluster automorphisms $\sigma_p$}
As observed in \cite{Sherman}, the transformations $\{\sigma_p\}_{p\in \mathbb{Z}}$ of $\calC(a,b)$ defined by 
$$\sigma_p:y_n\mapsto y_{2p-n}$$
 preserve the relations \eqref{fr} and define automorphisms of $\calC(a,b)$ for arbitrary parameters $a,b$.
 
  For example,
 $\sigma_2,\sigma_3$ send $y_1,y_2,y_3,y_4$ respectively onto  $y_0,y_1,y_2,y_3$ and $y_2,y_3,y_4,y_5$, 
 confirming that $\{y_0,y_1,y_2,y_3\}$ and $\{y_2,y_3,y_4,y_5\}$ are also sets of generators.

 Using the identities $y_1^b+1=y_1^b(y_3^b+1)-(y_1y_3-1)\sum_{i=0}^{b-1} (y_1y_3)^i$ and $y_4^a+1=y_4^a(y_2^a+1)-(y_2y_4-1)\sum_{i=0}^{a-1} (y_2y_4)^i$, we obtain in $\calC(a,b)$ the equalities
 
  $$\begin{array}{rclcl}
  y_0&=&\frac{y_1^b+1}{y_2}&=&y_1^by_4-y_2^{a-1}\sum_{i=0}^{b-1} (y_1y_3)^i,\\
  y_5&=&\frac{y_4^a+1}{y_3}&=&y_4^ay_1-y_3^{b-1}\sum_{i=0}^{b-1} (y_2y_4)^i.\vphantom{\Big)}
  \end{array}$$
 Hence, $\sigma_2,\sigma_3\in \Aut(\calC(a,b))$ correspond respectively to the following automorphisms of $\Spec \calC(a,b)$ (we again write $\sigma_i$ the dual geometric action induced by $\sigma_i$, and will in fact only work with this latter)
 $$\begin{array}{rrcl}
 \sigma_2\colon &(y_1,y_2,y_3,y_4)&\mapsto &(y_3,y_2,y_1,y_1^by_4-y_2^{a-1}\sum_{i=0}^{b-1} (y_1y_3)^i),\\
 \sigma_3\colon &(y_1,y_2,y_3,y_4)&\mapsto &(y_4^ay_1-y_3^{b-1}\sum_{i=0}^{b-1} (y_2y_4)^i,y_4,y_3,y_2).
 \end{array}$$
It is immediately checked that  $\sigma_2,\sigma_3$ satisfy $\sigma_2^2 = \sigma_3^2 = 1$ and generate a finite dihedral group $D_{2n}$ of order $2n$ or infinite dihedral group $D_\infty \cong \mathbb{Z}/2\mathbb{Z}\star \mathbb{Z}/2\mathbb{Z}=\mathbb{Z}\rtimes \mathbb{Z}/2\mathbb{Z}$.  The periodicity of the set of cluster variables easily implies that 
 $$n = \begin{cases}
      10,&ab = 1, \\
      6,&ab = 2,\\
      8,&ab = 3,\\
      \infty, & ab > 3.
\end{cases} 
 $$
 \subsection{The results}
 Our description of the group $ \Aut(\calC(a,b))\simeq\Aut(\Spec\calC(a,b))$, is given by the description of  
 $\Aut(\Spec\calC(a,b)\otimes \bbQ)$ with geometric tools, and by observing that the generators of the automorphism groups are defined over $\bbZ$, hence both groups are equal.
 
 In fact, we will precisely describe the group structure of the group $$\Aut(\Spec\calC(a,b)\otimes \k)$$ for any field  $\k$ of characteristic $0$ or characteristic prime to $ab$. This is the group of automorphisms of the affine surface over $\k$
 $$X(a,b) = \Spec \calC(a,b)\otimes \k.$$ 
 %In the following theorem, which summarizes our results, the group $D_{2n}$ is the dihedral group of $2n$ elements and %$D_{\infty}$ is the infinite dihedral group $\mathbb{Z}/2\mathbb{Z}\star \mathbb{Z}/2\mathbb{Z}=\mathbb{Z}\rtimes \mathbb{Z}/%2\mathbb{Z}$.
 
 \begin{theorem}\label{TheThm} The group 
 $$\mu_{a,b}:=\{(\mu,\nu)\in \k^*\mid \mu^a=\nu^b=1\}\subset \k^*{}^2$$
 acts on $X(a,b)$ as
 $$(y_1,y_2,y_3,y_4)\mapsto (\nu^{-1}y_1,\mu y_2,\nu y_3,\mu ^{-1}y_4).$$
 If $a=b$, then there is a group $H_{a,a}\subset \Aut(X(a,b))$ of order $2$, acting on $X(a,b)$ via
 $$\sigma_{5/2}\colon (y_1,y_2,y_3,y_4)\mapsto (y_4,y_3,y_2,y_1).$$
 We have
 $$\Aut(X(a,b))\cong\left\{\begin{array}{lll}
 \langle \sigma_2,\sigma_3\rangle &\simeq D_{10}&\mbox{ if }(a,b)=(1,1),\\
 \langle \sigma_2,\sigma_3\rangle \times\mu_{2,1}&\simeq D_6\times \mu_{2,1}\simeq D_{12} &\mbox { if } (a,b)=(2,1),\\
 \langle \sigma_2,\sigma_3\rangle \ltimes\mu_{3,1}&\simeq D_8\ltimes \mu_{3,1} &\mbox { if } (a,b)=(3,1),\\
( \langle \sigma_2,\sigma_3\rangle \ltimes\mu_{a,a})\rtimes H_{a,a}&\simeq ( D_{\infty} \ltimes\mu_{a,a})\rtimes \mathbb{Z}/2\mathbb{Z} &\mbox { if } a=b\ge 2,\\
\langle \sigma_2,\sigma_3\rangle \ltimes\mu_{a,b}&\simeq  D_{\infty} \ltimes\mu_{a,b} &\mbox { if } a\not= b, ab\ge 4.\\
 \end{array}\right.$$
 \end{theorem}
 Note that $\mu_{a,b}$ is isomorphic to $\mathbb{Z}/a\mathbb{Z}\times \mathbb{Z}/b\mathbb{Z}$ if $\k$ is algebraically closed, but is smaller in general. The group $\mu_{a,b}$ is the diagonalizable commutative algebraic group with the group of characters isomorphic to the abelian group $N$ corresponding to the seed skew-symmetrizable matrix  defining a cluster algebra. It is always a part of its  automorphism group and corresponds to its grading by the group $N$.\footnote{We owe this remark to Greg Muller.}
 
 The proof of Theorem~\ref{TheThm} is given by five propositions. More precisely, Proposition~\ref{Prop:Xa14} gives the cases $X(a,1)$, $a\ge 4$, Propositions~\ref{A2},~\ref{B2},~\ref{G2} give respectively the cases $A_2,B_2,G_2$, which correspond to $X(1,1),X(2,1),X(3,1)$, and the general case $X(a,b)$ with $a,b\ge 2$ is done in Proposition~\ref{GenCase}. 
 
 An automorphism of the $\bbZ$-algebra $\calC(a,b)$ is called a  \emph{cluster automorphism} if it sends each cluster to a cluster (see \cite{Assem}). Examples of such automorphisms are the automorphisms $\sigma_p$.  It follows from Theorem~\ref{TheThm} that the group of cluster  automorphisms of $\calC(a,b)$ is generated by 
 $\sigma_2$ and $\sigma_3$, except when $a=b\not=1$ in which case $\sigma_{5/2}\in H_{a,a}$ is a cluster automorphism not generated by $\sigma_2,\sigma_3$. In the case $(a,b) = (1,1)$, we have $\sigma_{5/2}=\sigma_2\sigma_3\sigma_2\sigma_3\sigma_2$ (see Remark~\ref{Rem:X11sigma52}).

 \medskip
 We thank Sergey Fomin and Greg Muller for coaching the second  author in the rudiments of the theory of cluster algebras. Thanks also to the referees for their constructive remarks that helped to improve the text.

\section{Compactifications with a pentagon}
In the sequel, all algebraic varieties are defined over a field $\k$ of characteristic zero or of characteristic $p$ with $ab\not\equiv 0 \pmod{p}$.

\begin{proposition}\label{Prop:TriangleBarX} The surface $X(a,b)$  admits a smooth compactification $\bar{X}(a,b)$ that is isomorphic to the blow-up of $\bbP^2$ at the points $$\{[1:0:\xi]\mid \xi^b+1 = 0\} \mbox{ and } \{[1:\lambda:0]\mid \lambda^a+1 = 0\}.$$ The boundary is the strict transform of the union of the coordinate lines described by the following picture $($where the numbers indicate the self-intersections of the irreducible components of the boundary$)$.

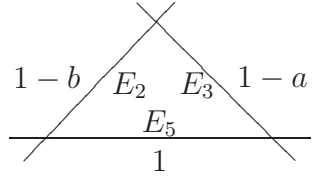
\begin{center}
\xy (-50,10)*{};
(0,0)*{};(40,0)*{}**\dir{-};(2,-3)*{};(22,18)*{}**\dir{-};(38,-3)*{};(17,18)*{}**\dir{-};
(20,-3)*{1};(5,8)*{1-b};(35,8)*{1-a};
(16,7)*{E_2};(25,7)*{E_3};(20,2)*{E_5};
\endxy
\end{center}
Moreover, the boundary divisor  is an anti-canonical divisor.
\end{proposition}
\begin{remark}
In the above proposition, the points are taken in a finite Galois extension $K$ of $\k$ (this works because $ab\not=0$ in $\k$). The blow-up map  is defined over $\k$, because the Galois group preserves the set of blown-up points. Also, each of the irreducible components of the boundary  is defined over $\k$. Note that the choice of the names of the curves here could seem strange to the reader, but it motivated by the sequel (see Corollary~\ref{Cor:Pentagon}).
\end{remark}

\begin{proof} Consider the projection map
$$\pi:X(a,b)\to \bbA^2, \quad (y_1,y_2,y_3,y_4)\mapsto (y_2,y_3).$$
The preimage of a point $(y_2,y_3)$ corresponds to points $(y_1,y_2,y_3,y_4)$ where 
$y_1y_3=y_2^a+1$ and $y_2y_4=y_3^b+1$. In particular $\pi$ restricts to an isomorphism $U\to V$, where $U$ and $V$ are the open subsets of $X(a,b)$ and $\A^2=\Spec(\k[y_2,y_3])$ given by $y_2y_3\not=0$. However, each point of the set 
$$\Delta=\{(0,\xi)\mid \xi^b+1 = 0\} \cup \{(\lambda,0)\mid \lambda^a+1 = 0\}\subset \A^2$$ 
has a preimage which is isomorphic to an affine line. Let $\eta\colon Z\to \A^2$ be the blow-up of $\bbA^2$ at points from $\Delta$. It remains to show that $\varphi=\eta^{-1}\circ \pi$ is an open embedding of $X(a,b)$ into $Z$, whose restriction is an isomorphism $X(a,b)\to Z\setminus (E_2\cup E_3)$, where $E_2$, $E_3$ are respectively the strict transforms of the lines of equation $y_2=0$ and $y_3=0$.

We first restrict ourselves to the open subsets $U_2\subset X(a,b)$ and $V_2\subset \A^2$ where $y_2\not=0$. The restriction of $\eta$ is then the blow-up of the ideal $(y_2^a+1,y_3)$, that we can write as
$$\eta^{-1}(V_2)=\{((y_2,y_3),[u:v])\in V_2\times \p^1\mid y_3 v=(y_2^a+1)u\},$$ where $\eta$ corresponds to the projection on the first factor. The map $\eta^{-1}\circ \pi$ sends thus $(y_1,y_2,y_3,y_4)$ onto $((y_2,y_3), [y_1:1])$. As the curve $v=0$ corresponds to $E_3$, we obtain an isomorphism $U_2\to \eta^{-1}(V_2)\setminus E_3$. Exchanging coordinates $y_2$, $y_3$, we obtain the same result when $y_3\not=0$. Since $(y_2,y_3)\not=(0,0)$ on $X(a,b)$, this gives the result.

 Since the anti-canonical class of $\bbP^2$ is 
represented by the the union of the coordinate lines, the strict transform of this divisor at its simple points is the anti-canonical divisor of the blow-up.
\end{proof}

\begin{corollary}\label{Cor:Pentagon} The surface $X(a,b)$ admits a smooth compactification $Z(a,b)$ with an anticanonical boundary described by the following picture

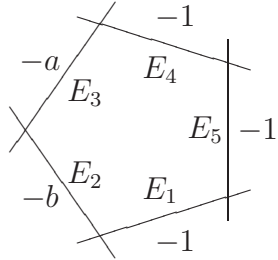
\begin{center}
\xy (-70,10)*{};
(-17,-3)*{};(-3,17)*{}**\dir{-};(-13,9)*{-a};(-7,5)*{E_3};(-8,15)*{};(15,8)*{}**\dir{-};(5,15)*{-1};(3,8)*{E_4};(12,12)*{};(12,-12)*{}**\dir{-};(16,0)*{-1};(9,0)*{E_5};(15,-8)*{};(-8,-15)*{}**\dir{-};(5,-15)*{-1};(3,-8)*{E_1};(-3,-17)*{};(-17,3)*{}**\dir{-};(-13,-9)*{-b};(-7,-6)*{E_2};\endxy
\end{center}
Moreover, $Z(a,b)$ is obtained by blowing-up the points $$\{[1:0:\xi]\mid \xi^b+1 = 0\} ,  \{[1:\lambda:0]\mid \lambda^a+1 = 0\}, [0:1:0], [0:0:1]$$
of $\mathbb{P}^2$.
\end{corollary}
\begin{proof}It suffices to blow-up the two points $E_5\cap E_3$ and $E_5\cap E_2$ from the compactification of Proposition~\ref{Prop:TriangleBarX}, which correspond to $[0:1:0]$ and $[0:0:1]$.\end{proof}
\begin{remark} 
The boundary $B=Z(a,b)\setminus X(a,b)$ being anti-canonical, every curve $C\subset Z(a,b)$ that is not contained in $B$ intersects the anti-canonical divisor $-K$ non-negatively. If $C$ is an irreducible curve of $B$, then $C^2+CK=-2$, which implies that $C(-K)=C^2+2$. This shows that $Z(a,b)$ is a weak Del Pezzo surface $(-K_X$ is big and nef$)$ if and only if $a,b\le 2$.

%Let $Y$ be a complete smooth surface and $D\in |-K_Y|$ be an anti-canonical divisor such that $Y\setminus D$ is affine. For any irreducible curve $C$ on $Y$ the intersection $K_Y\cdot C = -D\cdot C > 0$, unless $C$ is an irreducible component of $D$. Let $E$ be a $(-1)$-curve on $Y$. If $E$ is a component of $D$, we can blow it down to assume that $D$ does not contain such curves. If $E$ is not a component of $D$, then $D\cdot K_Y = -1$ implies that $E$ meets $D$ at one smooth point of $D$. Let $f:Y\to Y'$ be the blow-down of $E$. Then $D' = f(D)\in |-K_{Y'}|$. Continuing in this way we arrive at a surface $Z$ such that $-K_Z$ is ample, i.e. $Z$ is a del Pezzo surface. This shows that all smooth compactifications of $X(a,b)$ are obtained from blowing up of a del Pezzo surface with the boundary a strict transform of its anti-canonical divisor. When $a,b\le 3$ our surfaces $\bar{X}(a,b)$ are del Pezzo surfaces or weak del Del Pezzo surfaces (i.e. the anti-canonical divisor is not ample but big and nef). 
\end{remark}

\begin{remark} If $a\ge 2$ and $b\ge 2$, another natural normal compactification of the surface $X(a,b)$ is a complete intersection $Y(a,b)$ of two surfaces
$$x_1x_3-x_2^a-x_0^a = x_2x_4-x_3^b-x_0^b = 0$$
of degrees $a$ and $b$ in the weighted projective space $\bbP(1,a-1,1,1,b-1)$. The surface is singular at the point 
$[0,1,0,0,0]$ and $[0,0,0,0,1]$. Via the projection to the coordinates $x_0,x_2,x_3$, it  admits a birational map onto $\bbP^2$. In general, the compactification $Y(a,b)$ is related to the compactification $\bar{X}(a,b)$ in a rather complicated way by a sequence of blow-ups at singular points and then blow-downs. For example, when $a = b$, it is enough to resolve the singular points which are quotient singularities of type $\frac{1}{a}(1,1)$, and then blow down the strict transform of the line $x_0 = x_2 = x_3 = 0$. 
\end{remark}

%$$\calC(a,b) \cong \bbZ[y_1,y_2,y_3,y_4]/(y_1y_3-y_2^a\ -1,y_2y_4-y_3^b-1).$$

\section{Birational maps between $n$-gons pairs}
 \begin{definition}
 Let $Y$ be a smooth projective surface, and let $n\ge 2$ be an integer. A \emph{$n$-gon} on $Y$ is a divisor $B=E_1+\dots+E_n$, where the $E_i$ are curves of $Y$ isomorphic to $\p^1$, such that $$E_i\cdot E_j=\left\{\begin{array}{ll}1& \mbox{ if }|i-j|\in \{1,n-1\},\\
 0 & \mbox{ if }|i-j|\in \{2,\dots,n-2\}. \end{array}\right.$$

The pair $(Y,B)$ will be called \emph{$n$-gon pair}. The type of $B$ (or of the pair $(Y,B))$ is the sequence $(E_1^2,\dots,E_n^2)$, which is a $n$-uple of integers, defined up to cyclic permutation and reversion.

 We say that the $n$-gon $B$ (or the pair $(Y,B))$ is \emph{standard} if $n\ge 3$, and if there is an ordering of the $E_i$ such that $(E_1)^2=(E_2)^2=0$, $E_1\cdot E_2=1$ and $(E_i)^2\le -2$ for $i\ge 3$.
 \end{definition} 
 \begin{example}
Corollary~\ref{Cor:Pentagon} gives examples of pentagons ($n$-gons with $n=5$) which are not standard. In the sequel, we will use this example to provide either quadrangles or triangles, in a standard form.
 \end{example}
 \begin{definition}
 Let $(Y,B)$ be a $n$-gon pair and $(Y',B')$ be a $m$-gon pair. A birational map $f\colon Y\dasharrow Y'$ is called \emph{birational map of pairs} if it restricts to an isomorphism $Y\setminus B\to Y'\setminus B'$. If the map $f$ is regular (resp. biregular), we will moreover say that $f$ is a birational morphism of pairs (respectively an isomorphism of pairs).
 \end{definition}

 \begin{example}\label{Exa:Ellink}
 Let $(Y,B)$ be a $n$-gon pair of type $(0,0,-a,-b_1,\dots,-b_{n-3})$.
 
  Blowing-up the intersection point of the first two curves and contracting the strict transform of the second curve, we obtain a $n$-gon of type $(-1,0,-a+1,-b_1,\dots,-b_{n-3})$. 
  
  Blowing-up again the intersection of the first two curves and contracting  the strict transform of the second curve, we obtain a $n$-gon of type $(-2,0,-a+2,-b_1,\dots,-b_{n-3})$.
  
  After $a-1$ steps, we obtain a birational map of pairs $(Y,B)\dasharrow (Y',B')$, where $(Y',B')$ is a $n$-gon pair of type $(-a,0,0,-b_1,\dots,-b_{n-3})$. 
 
 
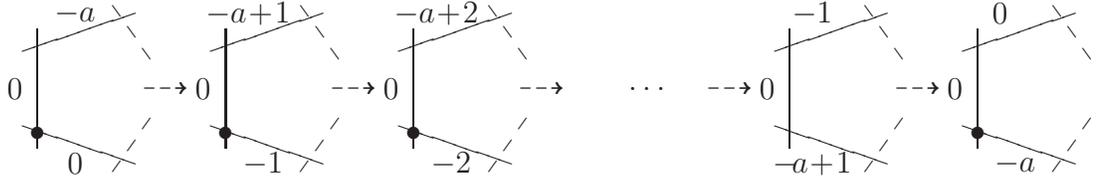
\begin{center}
\xy (0,15)*{};
 (7,4)*{};(2,11)*{}**\dir{--};(9,6)*{};
 (5,10)*{};(-10,5)*{}**\dir{-};(-3,10)*{-a};
 (-8,8)*{};(-8,-8)*{}**\dir{-};(-11,0)*{0};
 (-10,-5)*{};(5,-10)*{}**\dir{-};(-3,-10)*{0};
 (2,-11)*{};(7,-4)*{}**\dir{--};(-8,-6)*{\bullet};
 (9,0)*{\dasharrow};
 (32,4)*{};(27,11)*{}**\dir{--};(34,6)*{};
 (30,10)*{};(15,5)*{}**\dir{-};(20,10)*{-a\!+\!1};
 (17,8)*{};(17,-8)*{}**\dir{-};(14,0)*{0};
 (15,-5)*{};(30,-10)*{}**\dir{-};(22,-10)*{-1};
 (27,-11)*{};(32,-4)*{}**\dir{--};(17,-6)*{\bullet};
 (34,0)*{\dasharrow};
 (57,4)*{};(52,11)*{}**\dir{--};(59,6)*{};
 (55,10)*{};(40,5)*{}**\dir{-};(45,10)*{-a\!+\!2};
 (42,8)*{};(42,-8)*{}**\dir{-};(39,0)*{0};
 (40,-5)*{};(55,-10)*{}**\dir{-};(47,-10)*{-2};
 (52,-11)*{};(57,-4)*{}**\dir{--};(42,-6)*{\bullet};
 (59,0)*{\dasharrow};
 (84,0)*{\dasharrow};
 (73,0)*{\dots};
 (107,4)*{};(102,11)*{}**\dir{--};(109,6)*{};
 (105,10)*{};(90,5)*{}**\dir{-};(95,10)*{-1};
 (92,8)*{};(92,-8)*{}**\dir{-};(89,0)*{0};
 (90,-5)*{};(105,-10)*{}**\dir{-};(95,-10)*{-\!a\!+\!1};
 (102,-11)*{};(107,-4)*{}**\dir{--};
 (109,0)*{\dasharrow};
 (132,4)*{};(127,11)*{}**\dir{--};(134,6)*{};
 (130,10)*{};(115,5)*{}**\dir{-};(120,10)*{0};
 (117,8)*{};(117,-8)*{}**\dir{-};(114,0)*{0};
 (115,-5)*{};(130,-10)*{}**\dir{-};(122,-10)*{-a};
 (127,-11)*{};(132,-4)*{}**\dir{--};(117,-6)*{\bullet};
 \endxy
 \end{center}
 \end{example}
 \begin{definition}
 The birational maps $f\colon Y\dasharrow Y'$ as in Example~\ref{Exa:Ellink} will be called \emph{fibered modifications}.
 \end{definition}
 \begin{remark}
In Example~\ref{Exa:Ellink}, if $Y$ is a rational surface, then the linear system associated to the second $(0)$-curves\footnote{For any integer $m$, an $(m)$-curve is  a smooth rational curve with self-intersection equal to $m$.} of $B$ and $B'$ induce morphisms $\pi \colon Y\to \p^1$ and $\pi'\colon Y'\to \p^1$ with general fibre isomorphic to $\p^1$.  There exists then an automorphism $\theta$ of $\p^1$ such that the following diagramm commutes
$$\xymatrix{Y\ar[d]_{\pi}\ar@{-->}[r]^{f}& Y'\ar[d]^{\pi'}\\
\p^1\ar[r]^{\theta} &\p^1.
}$$ 

Note that the restriction of $\pi$ and $\pi'$ on the surfaces $Y\setminus B$ and  $Y'\setminus B'$ yield fibrations with general fibres isomorphic to $\A^1\setminus \{0\}$.
Hence, $f$ restricts to an isomorphism of  fibered surfaces $Y\setminus B\to Y'\setminus B'$. 
This explains why we call $f$ a fibered modification. \end{remark}
 \begin{proposition}\label{Prop:Ngons}
 Let $(Y,B)$ be a standard $n$-gon pair, and let $(Y',B')$ be a standard $m$-gon pair. 
 
Any birational map of pairs $f\colon (Y,B)\dasharrow (Y',B')$ decomposes into $$f=f_k\circ \dots \circ f_1,$$ where each $f_i\colon (Y_i,B_i)\dasharrow (Y_{i+1},B_{i+1})$ is either an isomorphism of pairs or a fibered modification $($and where $Y_0=Y,Y_k=Y'$, $B_0=B,B_k=B')$.

In particular, $m=n$.
 \end{proposition}
 \begin{proof}We can assume that $f$ is not an isomorphism.
 Let us take a minimal resolution 
$$\xymatrix{& Z\ar[ld]_{\pi}\ar[rd]^{\eta}\\
Y\ar@{-->}[rr]^{f} &&Y' 
}$$
of $f$. The fact that $f$ is a birational map of pairs $(X,B)\dasharrow (X',B')$ implies that the base-points of $f$ and $f^{-1}$ are in $B$ and $B'$ and that the restrictions of $\pi$ and $\eta$ yield isomorphisms $Z\setminus B_Z\to Y\setminus B$ and $Z\setminus B_Z\to X'\setminus B'$, for some divisor $B_Z$ on $Z$. Since $B$ and $B'$ are cycles, there is one cycle in $B_Z$, plus a priori some branches, which are then contracted by $\eta$ and $\pi$. By the minimality condition, this implies that $B_Z$ is in fact a cycle. In particular, the indeterminacy points of $f$ and $f^{-1}$ are singular points of $B$ and $B'$ respectively.

Let us  observe that $\eta$ contracts exactly one $(-1)$-curve. Firstly, the map $\eta$ is not an isomorphism because $f$ is not an isomorphism and because there is no $(-1)$-curve on $B'$. Secondly, if $\eta$ contracts at least two $(-1)$-curves, these are the strict transforms of the two $(0)$-curves of $B$. Hence the intersection point of these curves is blown-up by $\pi$ and the exceptional divisor of the point is sent by $f$ onto a curve of self-intersection $\ge 1$ of $B'$. However, $B'$ does not contain such curves. 

The same argument for $f^{-1}$ implies that $\pi$ also contracts exactly one $(-1)$-curve.
Hence, $f$ has a unique proper indeterminacy point $q$, and $\pi$ is a \emph{tower-resolution}, namely a sequence of blow-ups such that each point blown-up belongs to the exceptional curve of the previous point. In other words, $\pi$ is the blow-up of a chain of infinitely near points $x_r\succ x_{r-1}\succ\cdots\succ x_1 = q$. The $(-1)$-curve contracted by $\eta$ is the strict transform of one $(0)$-curve $E$ of $B$. Hence, $q$ is a singular point of $B$, lying on $E$. Note that $q$ is the intersection point of $E$ with the other $(0)$-curve $F$, since otherwise $F$ would be sent by $f$ onto a curve of self-intersection $\ge 1$. 

We denote by $(0,0,-a,-b_1,\dots,-b_{n-3})$ the type of $B$, where $F$ and $E$ correspond to the first and second curve respectively.
Blowing-up $q$ and contracting the strict transform of $E$, we obtain a birational map $\theta_1\colon (Y,B)\dasharrow (Y_{1},B_1)$, where $B_{1}$ is a $n$-gon of type 
 $(-1,0,-a+1,-b_1,\dots,-b_{n-3})$. The map $\varphi_1=f\circ \theta_1^{-1}\colon (Y_1,B_1)\dasharrow (Y',B')$ has one indeterminacy point less than $f$ (including in the counting all infinitely near points). Moreover, the fact that the minimal resolution $\pi$ of $f$ is a tower-resolution implies that $\varphi_{1}$ has a unique proper (i.e. not infinitely near) indeterminacy point $q_1$, which is the intersection of the first two curves. We write $\pi_{1}\colon Z_{1}\to Y_{1}$ the minimal resolution of $\varphi_{1}$, which is a again a tower-resolution, and denote by $\eta_{1}\colon Z_{1}\to Y'$ the birational morphism $\varphi_{1}\circ \pi_{1}$.
 
 Since $B_{1}$ contains exactly one curve of self-intersection $\ge -1$, $\eta_{1}$ contracts only one $(-1)$-curve, which is the strict transform of the $(0)$-curve of $B_{1}$. Blowing-up $q_1$ and contracting the strict transform of the $(0)$-curve, we obtain a birational map $\theta_2\colon (Y_{1},B_1)\dasharrow (Y_{2},B_2)$, where $B_{2}$ is a $n$-gon of type 
 $(-2,0,-a+2,-b_1,\dots,-b_{n-3})$. After $a-1$ steps, we obtain a pair $(Y_{a-1},B_{a-1})$, where $B_{a-1}$ is a $n$-gon of type 
 $(-a+1,0,-1,-b_1,\dots,-b_{n-3})$ on a smooth projective surface $Y_{a-1}$, and the birational map $\varphi_{a-1}=f\circ(\theta_{a-1}\circ \dots \circ \theta_1)^{-1}$ has $a-1$ indeterminacy points less than $f$. The unique indeterminacy point of $\varphi_{a-1}$ is the intersection point $q_{a-1}$ of the first two curves, but now the unique $(-1)$-curve contracted by $\eta_{a-1}$ is the strict transform of either the second or the the third curve. Blowing-up $q_{a-1}$ and contracting anyway the strict transform of the $(0)$-curve, we obtain a birational map $\theta_{a}\colon Y_{a-1}\dasharrow Y_{a}$ such that $\varphi_{a}=\varphi_{a-1}\circ (\theta_a)^{-1}=f\circ(\theta_{a}\circ \dots \circ \theta_1)^{-1}\colon (Y_{a},B_a)\dasharrow (Y',B')$ has either $a-1$ or $a$ indeterminacy points less than $\varphi$. Since $\theta_a\circ \dots \circ \theta_1$ is a fibered modification, the result follows by induction on the number of indeterminacy points. \end{proof}
 
\section{Compactifications of $X(a,1)$ with a triangle}\label{SecTriangle}
 The case of $X(a,1)$ is a bit different from the general case of $X(a,b)$ with $a,b\ge 2$, since the curve $E_2\subset Z(a,1)$ has self-intersection $-1$. 
 Denote by $\eta\colon (Z(a,1),B_Z)\to (T(a,1),B_T)$ the birational morphism of pairs which contracts the curves $E_2$ and $E_4$. The boundary $B_T$ of $X(a,1)$ in $T(a,1)$ consists of a triangle $\eta(E_1)+\eta(E_3)+\eta(E_5)={\mathcal{E}}_1+{\mathcal{E}}_3+{\mathcal{E}}_5$ of type $(0,-(a-2),0)$.

  
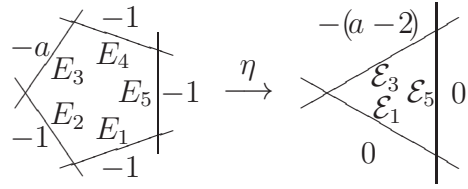
\begin{center}
\xy 
(-75,5)*{};
(-31,-2)*{};(-22,11)*{}**\dir{-};(-29,6)*{-a};(-25,10)*{};(-10,5)*{}**\dir{-};(-17,10)*{-1};(-12,8)*{};(-12,-8)*{}**\dir{-};(-9,0)*{-1};(-10,-5)*{};(-25,-10)*{}**\dir{-};(-17,-10)*{-1};(-22,-11)*{};(-31,2)*{}**\dir{-};(-29,-6)*{-1};
(-24,3)*{E_3};(-18,5)*{E_4};(-15,0)*{E_5};(-18,-5)*{E_1};(-24,-3)*{E_2};
(0,0)*{\longrightarrow};
(0,3)*{\eta};
(7,-2)*{};(28,10)*{}**\dir{-};(16,9)*{-(\!a-\!2\!)};(18.25,2)*{\mathcal{E}_3};(25,12)*{};(25,-12)*{}**\dir{-};(28,0)*{0};(23,0)*{\mathcal{E}_5};(28,-10)*{};(7,2)*{}**\dir{-};(16,-8)*{0};(18.5,-2.25)*{\mathcal{E}_1}; 
%(-70,5)*{};(-13,-2)*{};(8,10)*{}**\dir{-};(-4,8)*{0};(-1.75,2)*{\mathcal{E}_5};(5,12)*{};(5,-12)*{}**\dir{-};(8,0)*{0};(3,0)*{\mathcal{E}_3};(8,-10)*{};(-13,2)*{}**\dir{-};(-4,-8)*{0};(-1.5,-2.25)*{\mathcal{E}_1};  
 \endxy
 \end{center}
 
If $a\ge 4$, the triangle is a standard triangle, so the automorphisms of $X(a,1)$ can be described with the help of Proposition~\ref{Prop:Ngons}. The special cases where $a\le 3$ will then be described separately. The following lemma allows us to view $T(a,1)$ as a blow-up of points in $\p^1\times \p^1$.

 \begin{lemma}\label{Lemma:BlowUpTriangle}
 The smooth projective surface $T(a,1)$ is the blow-up $\pi'\colon T(a,1)\to \p^1\times \p^1$ of the following $a$ points 
 $$\{([\xi :1],[\xi:1])\mid\xi^a+1=0\}$$
and the boundary $T(a,1)\setminus X(a,1)$ consists of the strict transform of the curves $\pi'(\mathcal{E}_1)=\p^1\times [0:1]$, $\pi'(\mathcal{E}_5)=[1:0]\times\p^1$ and of the diagonal $\pi'(\mathcal{E}_3)$.
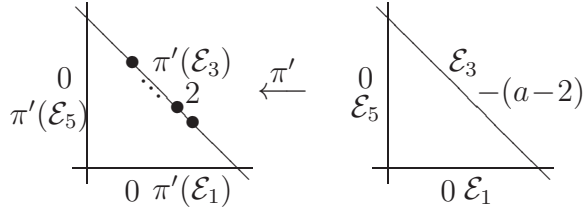
\begin{center}
\xy 
(-50,0)*{};
 (21.5,-1.5)*{};(-1.5,21.5)*{}**\dir{-};(14,10)*{2};(14,14)*{\pi'(\mathcal{E}_3)};
 (0,-2)*{};(0,22)*{}**\dir{-};(-5,7)*{\pi'(\mathcal{E}_5)};(-3,12)*{0};
 (-2,0)*{};(22,0)*{}**\dir{-};(6,-3)*{0};(14,-3)*{\pi'(\mathcal{E}_1)};(12,8)*{\bullet};(14,6)*{\bullet};(6,14)*{\bullet};(7.5,11.5)*{\cdot};(8.5,10.5)*{\cdot};(9.5,9.5)*{\cdot};
 (26,10)*{\longleftarrow}; (26,13)*{\pi'};
 (61.5,-1.5)*{};(38.5,21.5)*{}**\dir{-};(59,10)*{-(a\!-\!2)};(50,14)*{{\mathcal{E}}_3};
 (40,-2)*{};(40,22)*{}**\dir{-};(37,8)*{{\mathcal{E}}_5};(37,12)*{0};
 (38,0)*{};(62,0)*{}**\dir{-};(48,-3)*{0};(52,-3)*{{\mathcal{E}}_1};
 \endxy
 \end{center}
Moreover, the restriction of $\pi'$ to the affine surface $X(a,1)$ is given by 
$$(y_1,y_2,y_3,y_4)\mapsto  ([y_2:1],[1:y_4]).$$ 
 \end{lemma}
 \begin{remark}
 As in Proposition~\ref{Prop:TriangleBarX}, the points blown-up belong to a finite Galois extension of $\k$ so not necessarily to $\k$, but the morphism $\pi'$ is in any case defined over $\k$.
 \end{remark}
 \begin{proof}
 Recall that the birational morphism $\pi\colon Z(a,1)\to \p^2$ of Corollary~\ref{Cor:Pentagon} is the blow-up of the $a+3$ points
 $$[1:0:-1] ,  \{[1:\lambda:0]\mid \lambda^a+1 = 0\}, [0:1:0], [0:0:1].$$
 The birational morphism $\eta\colon Z(a,1)\to T(a,1)$ contracts the curves $E_2$, $E_4$, which are respectively the strict transform of the second coordinate line $\pi(E_2)$ and the exceptional divisor of $[0:1:0]$. Hence, denoting by $\kappa\colon \p^2\dasharrow \p^1\times \p^1$ the blow-up of $[0:0:1],[1:0:-1]$ followed by the contraction of the strict transform of $\pi(E_2)$, the map $\pi'=\kappa\pi\eta^{-1}\colon T(a,1)\to \p^1\times \p^1$ is the blow-up of the $a$ points $\{\kappa([1:\lambda:0])\mid \lambda^a+1 = 0\}$.
 Explicitely, we can choose $\kappa$ to be given by
 $$[x_0:x_1:x_2]\dasharrow ([x_1:x_0],[x_1:x_0+x_2]),$$
 which implies that $\{\kappa([1:\xi:0])\mid \xi^a+1 = 0\}=\{([\xi :1],[\xi:1])\mid\xi^a+1=0\}$. The restriction of $\pi\colon Z(a,1)\to \p^2$ being $(y_1,y_2,y_3,y_4)\mapsto (1:y_2:y_3)$ (see the proof of Proposition~\ref{Prop:TriangleBarX}), the restriction of $\pi'$ to $X(a,1)$ is 
 $$(y_1,y_2,y_3,y_4)\mapsto \kappa([1:y_2:y_3])=([y_2:1],[y_2:y_3+1])=([y_2:1],[1:y_4]).$$
 \end{proof}
 
 \begin{lemma}\label{Lem:AutoBoundA1}
The action of the group $\Aut(T(a,1),B_T)$ of automorphisms of the pair $(T(a,1),B_T)$ on the set $\{\mathcal{E}_1,\mathcal{E}_3,\mathcal{E}_5\}$ induces a split exact sequence 
$$1\longrightarrow \mu_{a,1} \to \Aut(T(a,1),B_T)\to R_a\to 1,$$
where $\mu_{a,1}\simeq \{\mu \in \k^*\mid \mu^a=1\}$ acts on $X(a,1)$ via
 $$(y_1,y_2,y_3,y_4)\mapsto (y_1,\mu y_2,y_3,\mu ^{-1}y_4),$$
 and where
 
 $$R_a=\left\{\begin{array}{lll}\langle \sigma_2,\sigma_3\rangle & \simeq \frakS_3,&\mbox{ if } a=2,\\
 \langle \sigma_2\rangle & \simeq \mathbb{Z}/2\mathbb{Z},&\mbox{ if } a\not=2.\\
 \end{array}
 \right.$$
 \end{lemma}
 \begin{proof}
 Denote by $\mu_{a,1}$ the kernel of the action of $\Aut(T(a,1),B_T)$ on the set $\{\mathcal{E}_1,\mathcal{E}_3,\mathcal{E}_5\}$. Let us observe that the set of $a$ curves contracted by $\pi'$ is invariant by $\mu_{a,1}$. Indeed, the image by $\mu_{a,1}$ of one of the curves is an irreducible curve, not intersecting $\mathcal{E}_1$ and $\mathcal{E}_5$. The image of this curve by $\pi'$ does not intersect the two fibres $\pi'(\mathcal{E}_1)=\p^1\times [0:1]$, $\pi'(\mathcal{E}_5)=[1:0]\times\p^1$, so it is a point.
 
 The group $\mu_{a,1}$ is then the lift of automorphisms of $\p^1\times\p^1$ which leave invariant the three curves $\pi'(\mathcal{E}_1)$, $\pi'(\mathcal{E}_3)$, $\pi'(\mathcal{E}_5)$ and which preserve the set $\{([\xi :1],[\xi:1])\mid\xi^a+1=0\}$. This group is isomorphic to $ \{\mu \in \k^*\mid \mu^a=1\}$, acts on  $\p^1\times \p^1$ via
$$([u_1:u_2],[v_1:v_2])\mapsto ([\mu u_1:u_2],[\mu v_1:v_2])$$
 and then on $X(a,1)=\Spec \k[y_1,y_2,y_3,y_4]/(y_1y_3-y_2^a\ -1,y_2y_4-y_3-1)$ via $$(y_1,y_2,y_3,y_4)\mapsto (y_1,\mu y_2,y_3,\mu ^{-1}y_4).$$
 It coincides then with the group $\mu_{a,1}$ already defined in the introduction.
 
The explicit description of $\pi'\colon T(a,1)\to \p^1\times \p^1$ given in Lemma~\ref{Lemma:BlowUpTriangle} shows that the automorphism 
$$([u_1:u_2],[v_1:v_2])\mapsto ([v_2:v_1],[u_2:u_1])$$
of $\p^1\times \p^1$ lifts to an automorphism of $T(a,1)$ which preserves the boundary, exchanging the two $(0)$-curves and preserving the $(-a)$-curve.
In affine coordinates, this gives the following automorphism of $X(a,1)\subset \A^4$
$$(y_1,y_2,y_3,y_4)\mapsto \left(\frac{y_4^a+1}{y_2y_4-1},y_4,y_3,y_2\right)=\left(y_1y_4^a-\sum_{i=0}^{a-1}(y_2y_4)^i,y_4,y_3,y_2\right),$$
which corresponds to the automorphism $\sigma_3$. If $a\not=2$, then this element generates the image of the action, since the three curves $\mathcal{E}_1,\mathcal{E}_3,\mathcal{E}_5$ have self-intersection $0,2-a,0$ respectively. 

If $a=2$, then $\pi'$ blows-up the two points $([\xi :1],[\xi:1])$, where $\xi^2+1=0$, which are the two indeterminacy points of the birational involution
$$([u_1:u_2],[v_1:v_2])\mapsto ([u_1:u_2],[u_1v_2-u_2v_1:u_1v_1+u_2v_2])$$
of $\p^1\times\p^1$. The lift of the involution gives an automorphism of $B(a,1)$, which fixes $\mathcal{E}_5$ and exchanges the two curves $\mathcal{E}_1$ and $\mathcal{E}_3$. This involution yields the automorphism of order $2$ of $X(2,1)$, given by 
$$\sigma_2\colon (y_1,y_2,y_3,y_4)\mapsto (y_3,y_2,y_1,y_1y_4-y_2).$$ \end{proof}
If $a\ge 4$, we can decompose any automorphism of $X(a,1)$ into a sequence of isomorphisms of pairs and fibered modifications (Proposition~\ref{Prop:Ngons}). A priori, the fibered modification could go from one pair to a different one, but we will show that in the case of $X(a,1)$, we can only consider fibered modifications $T(a,1)\dasharrow T(a,1)$, hence each of them can be seen as a unique automorphism of $X(a,1)$, up to automorphisms of the pair $X(a,1)\subset T(a,1)$ (and these latter automorphisms have been described in Lemma~\ref{Lem:AutoBoundA1}).

\begin{example}\label{Exa:Fibsigma3}
The following birational involution of $\p^1\times\p^1$
$$f\colon ([u_1:u_2],[v_1:v_2])\dashmapsto ([u_1:u_2],[u_2^{a-2}(u_1v_2-u_2v_1):u_1^{a-1}v_1 +u_2^{a-1}v_2])$$
is not defined only at $([1:0],[0:1])$ and at 
 $\{([\xi :1],[\xi:1])\mid\xi^a+1=0\}$. It follows from the explicit description that the lift of $f$ is a birational map $\hat{f}=(\pi')^{-1}f\pi'$ of $T(a,1)$ which restricts to an automorphism  of $T(a,1)\setminus {\mathcal{E}}_5$, and which exchanges $\mathcal{E}_1$ and $\mathcal{E}_3$. 
 
 The map $\hat{f}$ is therefore a fibered modification if $a\not=2$ and an isomorphism if $a=2$. Moreover, $\hat{f}$ restricts to an automorphism of $X(a,1)$, that we will show to be equal to $\sigma_2$.

To compute this, we use the map $X(a,1)\to \p^1\times \p^1$ given by $(y_1,y_2,y_3,y_4)\mapsto  ([y_2:1],[1:y_4])$. The composition with $f$ yields
$$
([y_2:1],[1:y_4])\dashmapsto ([y_2:1],[(y_2y_4-1):y_2^{a-1} +y_4])=([y_2:1],[1:y_1y_4-y_2^{a-1}]).$$
Hence, $y_3=y_2y_4-1$ is exchanged with $y_2(y_1y_4-y_2^{a-1})-1=y_1$. The involutive automorphism of $X(a,1)$ is thus given by 
$$\sigma_2\colon (y_1,y_2,y_3,y_4)\mapsto (y_3,y_2,y_1,y_1y_4-y_2^{a-1}).$$
\end{example}

\begin{proposition}\label{Prop:Xa14}
If $a\ge 4$, then $\Aut(X(a,1))=\mu_{a,1}\rtimes \langle \sigma_2,\sigma_3\rangle$. 
Moreover, $\langle \sigma_2,\sigma_3\rangle= \langle \sigma_2\rangle \star \langle \sigma_3\rangle= \mathbb{Z}/2\mathbb{Z} \star \mathbb{Z}/2\mathbb{Z}=\mathbb{Z}\rtimes \mathbb{Z}/2\mathbb{Z}$ is an infinite diedral group and $\mu_{a,1}$ is a finite cyclic group.
\end{proposition}
\begin{proof}
Because $a\ge 4$, the pair $(T(a,1),B_T)$ is standard. According to Proposition~\ref{Prop:Ngons}, every automorphism of $X(a,1)$ decomposes into fibered modification and isomorphisms of pairs. Each fibered modification is equal to $\sigma_2$, up to isomorphism of pairs (Example~\ref{Exa:Fibsigma3}), and each automorphism of the pair $(T(a,1),B_T)$ is generated by $\sigma_3$ and $\mu_{a,1}$ (Lemma~\ref{Lem:AutoBoundA1}). Hence, $\Aut(X(a,1))$ is generated by $\mu_{a,1}$, $\sigma_2$ and $\sigma_3$. To achieve the proof, it remains to observe that $\sigma_2\sigma_3$ is of infinite order. The map $\sigma_2\sigma_3$ and its inverse have both a unique proper indeterminacy point, and these two points are different. Proceeding by induction, we obtain that $(\sigma_2\sigma_3)^n$ has again a unique proper indeterminacy point for any $n\ge 1$, always being the proper indeterminacy point of $\sigma_2\sigma_3$.
\end{proof}

\section{Cluster algebras of types $A_2,B_2,G_2$}\label{abg}
By contrast to the case $a\ge 4$, we will see that the group of automorphisms of $X(a,1)$ with $a=1,2,3$ is finite, and is in fact contained in the group of automorphisms of a symmetric $n$-gon $(Y(a,1),B_Y)$ that  we define now.

\begin{enumerate}
\item
The pair $(Y(1,1),B_Y)=(Z(1,1),B_Z)$ is  a pentagon of five $(-1)$-curves.
 
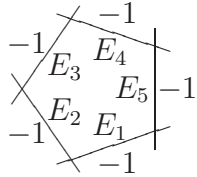
\begin{center}
\xy 
(-70,5)*{};
(-11,-2)*{};(-2,11)*{}**\dir{-};(-9,6)*{-1};(-5,10)*{};(10,5)*{}**\dir{-};(3,10)*{-1};(8,8)*{};(8,-8)*{}**\dir{-};(11,0)*{-1};(10,-5)*{};(-5,-10)*{}**\dir{-};(3,-10)*{-1};(-2,-11)*{};(-11,2)*{}**\dir{-};(-9,-6)*{-1};
(-4,3)*{E_3};(2,5)*{E_4};(5,0)*{E_5};(2,-5)*{E_1};(-4,-3)*{E_2};
 \endxy
 \end{center}
\item
The pair $(Y(2,1),B_Y)=(T(2,1),B_T)$ is a triangle of three $(0)$-curves.
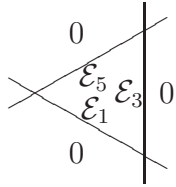
\begin{center}
\xy 
(-70,5)*{};
(-13,-2)*{};(8,10)*{}**\dir{-};(-4,8)*{0};(-1.75,2)*{\mathcal{E}_5};(5,12)*{};(5,-12)*{}**\dir{-};(8,0)*{0};(3,0)*{\mathcal{E}_3};(8,-10)*{};(-13,2)*{}**\dir{-};(-4,-8)*{0};(-1.5,-2.25)*{\mathcal{E}_1};  
 \endxy
 \end{center}
\item
The pair $(Y(3,1),B_Y)$ is obtained by blowing-up the point $\mathcal{E}_1\cap \mathcal{E}_5$ in $(T(3,1),B_T)$, and is a square of four $(-1)$-curves. We denote by $\mathcal{E}_i'$ the strict transform of $\mathcal{E}_i$ and by $\mathcal{E}_7'$ the exceptional curve produced, and obtain the following diagram.
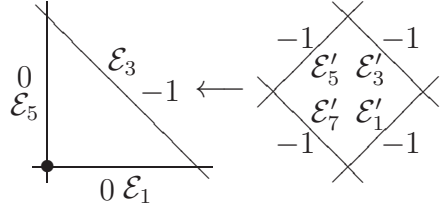
\begin{center}
\xy 
(-90,5)*{};
 (-18.5,-11.5)*{};(-41.5,11.5)*{}**\dir{-};(-25,0)*{-1};(-30,4)*{{\mathcal{E}}_3};
 (-40,-12)*{};(-40,12)*{}**\dir{-};(-43,-2)*{{\mathcal{E}}_5};(-43,2)*{0};
 (-42,-10)*{};(-18,-10)*{}**\dir{-};(-32,-13)*{0};(-28,-13)*{{\mathcal{E}}_1};(-40,-10)*{\bullet};
(-17,0)*{\longleftarrow};
(-12,-2)*{};(2,12)*{}**\dir{-};(-7,7)*{-1};(-3,3)*{\mathcal{E}_5'};(-2,12)*{};(12,-2)*{}**\dir{-};(7,7)*{-1};(3,3)*{\mathcal{E}_3'};(12,2)*{};(-2,-12)*{}**\dir{-};(7,-7)*{-1};(3,-3)*{\mathcal{E}_1'};(2,-12)*{};(-12,2)*{}**\dir{-};(-7,-7)*{-1};(-3,-3)*{\mathcal{E}_7'};
 \endxy
 \end{center}
\end{enumerate}
\begin{lemma}\label{Lem:ExtXa1}
For $a = 1,2,3$, every automorphism of $X(a,1)$ extends to an automorphism of the pair $(Y(a,1),B_Y)$.
\end{lemma}
\begin{proof}
Suppose for contradiction the existence of a birational map $$f\colon (Y(a,1),B_Y)\dasharrow (Y(a,1),B_Y)$$ which is not an automorphism of pairs. Recall that the type of $B_Y$ is 
\begin{center}$(-1,-1,-1,-1,-1)$, $(0,0,0)$ or $(-1,-1,-1,-1)$.\end{center}
 Let us take a minimal resolution 
$$\xymatrix{& Z\ar[ld]_{\pi}\ar[rd]^{\eta}\\
Y(a,1)\ar@{-->}[rr]^{f} &&Y(a,1)
}$$
of $f$. As observed in Proposition~\ref{Prop:Ngons}, it follows from the minimality condition that the preimage $B_Z$ of $B$ under $\pi$ is equal to the preimage of $B$ under $\eta$ and consists of a cycle of smooth rational curves. In particular, the indeterminacy points of $f$ and $f^{-1}$ are singular points of $B$. 

Since $f$ is not an isomorphism, $\pi$ and $\eta$ contract at least one $(-1)$-curve. 

If $\eta$ contracts at least two $(-1)$-curves, these are sent by $\pi$ onto two curves of $E_1,E_2\subset B$ of self-intersection $\ge -1$. If $E_1,E_2$ are $(-1)$-curves of $B$, $\pi$ does not blow-up any point of these two disjoint curves. There is one irreducible curve of $B$ touching these two curves, which is thus sent by $f$ onto a curve of self-intersection $\ge 0$; this is impossible. If $E_1$ and $E_2$ are $(0)$-curves, then $\pi$ blows-up the point of intersection, but no other point. This is impossible since the boundary obtained would have only two curves.

The only remaining case is when $\eta$ contracts exactly one $(-1)$-curve, and by symmetry we can also assume that $\pi$ also contracts one $(-1)$-curve, which implies that $f$ has exactly one proper indeterminacy point. If $B$ is a triangle, we observe that the image of the $(0)$-curve not touching the indeterminacy point is a curve of self-intersection $\ge 1$, which is impossible. The remaining case is when $B$ only contains $(-1)$-curves, and  so the $(-1)$-curve contracted by $\eta$ is the strict transform by $\pi$ of a $(-1)$-curve $E$ of $B$. No point of $E$ is then blown-up by $\pi$. Since $f$ has only one proper indeterminacy point, there is a $(-1)$-curve $E'$ of $B$ which touches $E$ and which does not contain any indeterminacy point. Its image by $f$ is a curve of self-intersection $\ge 0$, which is impossible. 
\end{proof}

 \begin{proposition}[Case $A_2$]\label{A2} The group $\Aut(X(1,1))$ is a dihedral group $D_{10}$ of order 10 generated by the cluster transformations $\sigma_2,\sigma_3$, which act on the pentagon $(E_1,\dots,E_5)$ via the following actions:
 
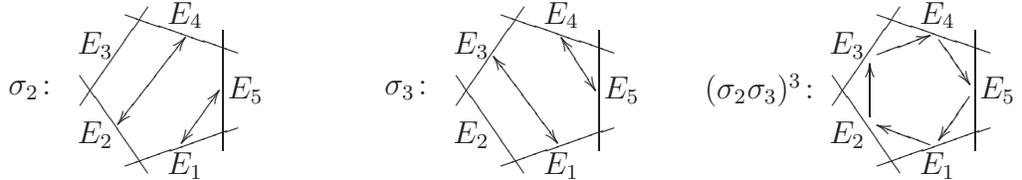
\begin{center}
\xy 
(-30,5)*{};
(-17,0)*{\sigma_2\colon };
(-11,-2)*{};(-2,11)*{}**\dir{-};(-9,6)*{E_3};(-5,10)*{};(10,5)*{}**\dir{-};(3,10)*{E_4};(8,8)*{};(8,-8)*{}**\dir{-};(11,0)*{E_5};(10,-5)*{};(-5,-10)*{}**\dir{-};(3,-10)*{E_1};(-2,-11)*{};(-11,2)*{}**\dir{-};(-9,-6)*{E_2};
{\ar@{<->} (2.5,-7)*{};(7.5,0)*{}};
{\ar@{<->} (3,7)*{};(-6,-4.5)*{}};
(33,0)*{\sigma_3\colon };
(39,-2)*{};(48,11)*{}**\dir{-};(41,6)*{E_3};(45,10)*{};(60,5)*{}**\dir{-};(53,10)*{E_4};(58,8)*{};(58,-8)*{}**\dir{-};(61,0)*{E_5};(60,-5)*{};(45,-10)*{}**\dir{-};(53,-10)*{E_1};(48,-11)*{};(39,2)*{}**\dir{-};(41,-6)*{E_2};
{\ar@{<->} (52.5,-7)*{};(44,4.5)*{}};
{\ar@{<->} (53,7)*{};(57.5,0)*{}};
(80,0)*{(\sigma_2\sigma_3)^3\colon };
(89,-2)*{};(98,11)*{}**\dir{-};(91,6)*{E_3};(95,10)*{};(110,5)*{}**\dir{-};(103,10)*{E_4};(108,8)*{};(108,-8)*{}**\dir{-};(111,0)*{E_5};(110,-5)*{};(95,-10)*{}**\dir{-};(103,-10)*{E_1};(98,-11)*{};(89,2)*{}**\dir{-};(91,-6)*{E_2};
{\ar (102,-7.3)*{};(95,-4.7)*{}};
{\ar (94,-3.5)*{};(94,3.5)*{}};
{\ar (95,4.7)*{};(102,7.3)*{}};
{\ar (103.2,6.7)*{};(107.3,0.8)*{}};
{\ar (107.3,-0.8)*{};(103.2,-6.7)*{}};
 \endxy
 \end{center}
\end{proposition}

\begin{remark}\label{Rem:X11sigma52}
The natural automorphism $\sigma_{5/2}\colon (y_1,y_2,y_3,y_4)\mapsto (y_4,y_3,y_2,y_1)$  of $X(1,1)$ corresponds to the permutation $E_1\leftrightarrow E_4$, $E_2\leftrightarrow E_3$ and is thus equal to $\sigma_2\sigma_3\sigma_2\sigma_3\sigma_2$.
\end{remark}
\begin{proof}According to Lemma~\ref{Lem:ExtXa1}, we have $\Aut(X(1,1))=\Aut(Y(1,1),B_Y)$. Recall that $Y(1,1)=Z(1,1)\to T(1,1)$ contracts $E_2$ and $E_4$. By Lemma~\ref{Lemma:BlowUpTriangle}, $\sigma_2$ extends to an automorphism of $(T(1,1),B_T)$ which exchanges $\mathcal{E}_1$ and $\mathcal{E}_5$ and which fixes $\mathcal{E}_3$. Hence, the corresponding automorphism of $Y(1,1)$ corresponds to the permutation $ E_1\leftrightarrow E_5,E_2\leftrightarrow E_4 $. The action of $\sigma_3$ is given in Example~\ref{Exa:Fibsigma3}. Since it exchanges $\mathcal{E}_1$ and $\mathcal{E}_3$, it corresponds to the permutation $ E_1\leftrightarrow E_3,E_4\leftrightarrow E_5$.

Observe that $\sigma_2\sigma_3$ acts as $E_1\to E_3\to E_5\to E_2\to E_4$, which implies that $(\sigma_2\sigma_3)^3$ is the permutation $E_1\to E_2\to E_3\to E_4\to E_5$. Since both $\sigma_2$ and $\sigma_3$ conjugate $(\sigma_2\sigma_3)^3$ to its inverse, the group generated by $\sigma_2,\sigma_3$ admits a surjective homomorphism to $D_{10}$.

Let us observe that $Y(1,1)$ is a del Pezzo surface of degree $5$, a fact which directly follows from the description of $\pi\colon Y(1,1)=Z(1,1)\to \p^1\times \p^1$ (see also Remark~\ref{CubicX11} for another argument).

It follows from the classification of automorphism groups of del Pezzo surfaces that $\Aut(Y(1,1)) \subset \frakS_5$, and equality holds if $\k$ is algebraically closed. The group $\frakS_5$ is the group of symmetries of the union of ten $(-1)$-curves on $Y(1,1)$ whose intersection graph is the Petersen graph (see \cite{CAG}).

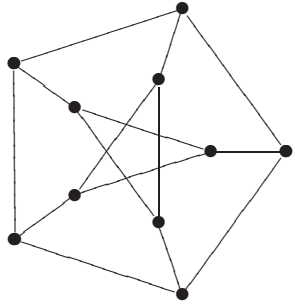
\xy (-70,25)*{};(-70,-25)*{};
(20,0)*{\bullet};(6.2,19)*{\bullet}**\dir{-};(6.2,19)*{};(-16.2, 11.7)*{\bullet}**\dir{-};(-16.2, 11.7)*{\bullet};(-16,-11.7)*{\bullet}**\dir{-};
(-16.2,-11.7)*{\bullet};(6.2,-19)*{\bullet}**\dir{-};(6.2,-19)*{};(20,0)*{\bullet}**\dir{-};
(10,0)*{\bullet};(-8.1,5.85)*{\bullet}**\dir{-};(-8.1,5.85)*{};(3.1, -9.5)*{\bullet}**\dir{-};(3.1, -9.5)*{};(3.1,9.5)*{\bullet}**\dir{-};
(3.1,9.5)*{};(-8.1,-5.85)*{\bullet}**\dir{-};
(-8.5,-5.85)*{};(10,0)*{}**\dir{-};   (10,0)*{};(20,0)*{}**\dir{-};  (3.1,9.5)*{};(6.2,19)*{}**\dir{-};
 (-16.2, 11.7)*{};(-8.1,5.85)*{}**\dir{-};(-16.2, -11.7)*{};(-8.2,-5.85)*{}**\dir{-};(6.2, -19)*{};(3.1,-9.85)*{}**\dir{-};
\endxy

 In the anti-canonical model, these $(-1)$ curves are the 10 lines on the surface. The boundary $B$ consists of five $(-1)$-curves forming a subgraph of the Petersen graph isomorphic to a pentagon. There are 12 such subgraphs, and the stabilizer group of each one is isomorphic to $D_{10}$. This shows that $\Aut(X(1,1))$ is contained in $D_{10}$, and hence coincides with it.
\end{proof}
\begin{remark}\label{CubicX11}
Note that, expressing $y_2$ in terms of $y_1$ and $y_3$, we obtain that   $X(1,1)$ admits a natural compactification in $\bbP^3$ isomorphic to a cubic surface $X$ with equation
 $$x_1x_2x_3 - x_0^2x_4-x_2x_0^2-x_0^3 = 0.$$
 The boundary $x_0 = 0$ consists of three coplanar lines and the surface has two of the intersection points as ordinary double points.   The intersection 
 points of these  lines are singular  points of the cubic surface. The points $[0:1:0:0]$ and $[0:0:1:0]$ are ordinary double points and the point 
 $[0:0:1:0]$ is a double rational point of type $A_2$. Let $X'\to X$ be a minimal resolution of singularities. The pre-image of the boundary is a 7-gon of type $(-1,-2,-2,-1,-2,-1,-2)$. By blowing down the first and fourth curve, we obtain a smooth compactification $Y(1,1)$ with the boundary equal to a pentagon of $(-1)$-curves.  Since $X'$ is a weak del Pezzo surface of degree 3, the surface $Y(1,1)$ is a del Pezzo surface of degree 5.

\end{remark}

\begin{proposition}[Case $B_2$] \label{B2} The group $\Aut(X(2,1))$ is isomorphic to $\frakS_3\times \mu_{2,1}\simeq D_{12}$. The group $\frakS_3$ is generated by $\sigma_2$ and $\sigma_3$, and $\mu_{2,1}$ by the automorphism $(y_1,y_2,y_3,y_4)\mapsto (y_1,-y_2,y_3,-y_4)$, which fixes the three curves $\mathcal{E}_1,\mathcal{E}_3,\mathcal{E}_5$. The actions of $\sigma_2$ and $\sigma_3$ on the triangle are the following
\begin{center}
\xy (0,15)*{};
(-50,5)*{};
(-17,0)*{\sigma_2\colon };
(-13,-2)*{};(8,10)*{}**\dir{-};(-4,8)*{\mathcal{E}_5};(5,12)*{};(5,-12)*{}**\dir{-};(8,0)*{\mathcal{E}_3};(8,-10)*{};(-13,2)*{}**\dir{-};(-4,-8)*{\mathcal{E}_1};
{\ar@{<->} (-1.75,-3.5)*{};(-1.75,3.5)*{}};
(33,0)*{\sigma_3\colon };
(37,-2)*{};(58,10)*{}**\dir{-};(46,8)*{\mathcal{E}_5};(55,12)*{};(55,-12)*{}**\dir{-};(58,0)*{\mathcal{E}_3};(58,-10)*{};(37,2)*{}**\dir{-};(46,-8)*{\mathcal{E}_1};
{\ar@{<->} (48.5,-3.5)*{};(54.5,0)*{}};
 \endxy
 \end{center}
\end{proposition}

\begin{proof}According to Lemma~\ref{Lem:ExtXa1},  $\Aut(X(2,1))=\Aut(Y(2,1),B_Y)=\Aut(T(2,1),B_T)$. The action of $\sigma_2$ and $\sigma_3$ on the triangle are given in Lemma~\ref{Lem:AutoBoundA1} and Example~\ref{Exa:Fibsigma3}. One can moreover check that $\sigma_2$ and $\sigma_3$ generate a group isomorphic to $\frakS_3$. By Lemma~\ref{Lem:AutoBoundA1}, we have a split exact sequence
$$1\to \mu_{2,1}\to \Aut(X(1,2))\to \frakS_3\to 1.$$
We can then easily check that $\sigma_2,\sigma_3$ commute with $\mu_{2,1}$. Hence, $\Aut(X(1,2))$ is isomorphic to $\frakS_3\times \mu_{2,1}$. In particular, $\mu_{2,1}\simeq \mathbb{Z}/2\mathbb{Z}$ and $\Aut(X(1,2))\simeq \frakS_3\times \mathbb{Z}/2\mathbb{Z}=D_{12}$.
\end{proof}

\begin{remark} Similarly to the previous case, the surface $X(2,1)$ admits a compactification $X$ isomorphic to a cubic surface
$$x_1x_2x_3-x_0^2x_4-x_2^2x_0-x_0^3 = 0.$$
The boundary $x_0 = 0$ consists of three coplanar lines. The surface has 2 singular points of types $A_2$ and $A_1$. We leave to the reader to find a birational isomorphism  from $X$ to our compactification $Y(2,1)$. The surface $Y(2,1)$ is a del Pezzo surface of degree 6. This latter observation also follows from the description of the morphism $Y(2,1)=T(2,1)\to \p^1\times \p^1$, which is the blow-up of two general points.
\end{remark}

\begin{proposition}[Case $G_2$]\label{G2}The group $\Aut(X(3,1))$ is isomorphic to $\mu_{3,1}\rtimes D_8$. It is generated by   the group of cluster automorphisms $D_8$ generated by $\sigma_2$ and $\sigma_3$, and by $\mu_{3,1}\simeq \{\mu \in \k^*\mid \mu^3=1\}$, acting on $X(3,1)$ via 
 $$(y_1,y_2,y_3,y_4)\mapsto (y_1,\mu y_2,y_3,\mu ^{-1}y_4),$$
 and fixing the four curves $\mathcal{E}_1',\mathcal{E}_3',\mathcal{E}_5',\mathcal{E}_7'$. The actions of $\sigma_2$ and $\sigma_3$ on the square are the following
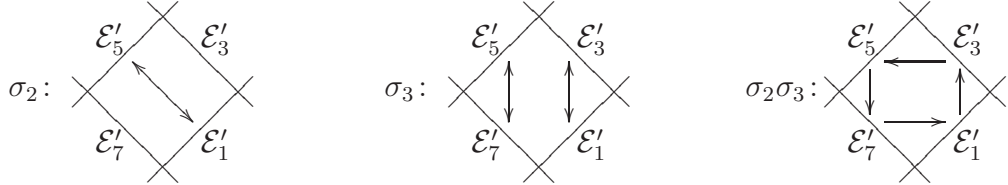
\begin{center}
\xy 
(-30,5)*{};
(-17,0)*{\sigma_2\colon };
(-12,-2)*{};(2,12)*{}**\dir{-};(-7,7)*{\mathcal{E}_5'};(-2,12)*{};(12,-2)*{}**\dir{-};(7,7)*{\mathcal{E}_3'};(12,2)*{};(-2,-12)*{}**\dir{-};(7,-7)*{\mathcal{E}_1'};(2,-12)*{};(-12,2)*{}**\dir{-};(-7,-7)*{\mathcal{E}_7'};
{\ar@{<->} (-4,4)*{};(4,-4)*{}};
(33,0)*{\sigma_3\colon };
(38,-2)*{};(52,12)*{}**\dir{-};(43,7)*{\mathcal{E}_5'};(48,12)*{};(62,-2)*{}**\dir{-};(57,7)*{\mathcal{E}_3'};(62,2)*{};(48,-12)*{}**\dir{-};(57,-7)*{\mathcal{E}_1'};(52,-12)*{};(38,2)*{}**\dir{-};(43,-7)*{\mathcal{E}_7'};
{\ar@{<->}(54,4)*{};(54,-4)*{}};
{\ar@{<->}(46,4)*{};(46,-4)*{}};
(83,0)*{\sigma_2\sigma_3\colon };
(88,-2)*{};(102,12)*{}**\dir{-};(93,7)*{\mathcal{E}_5'};(98,12)*{};(112,-2)*{}**\dir{-};(107,7)*{\mathcal{E}_3'};(112,2)*{};(98,-12)*{}**\dir{-};(107,-7)*{\mathcal{E}_1'};(102,-12)*{};(88,2)*{}**\dir{-};(93,-7)*{\mathcal{E}_7'};
{\ar(106,-3)*{};(106,3)*{}};
{\ar(104,4)*{};(96,4)*{}};
{\ar(94,3)*{};(94,-3)*{}};
{\ar(96,-4)*{};(104,-4)*{}};
 \endxy
 \end{center}
\end{proposition}
\begin{proof}
According to Lemma~\ref{Lem:ExtXa1}, we have $\Aut(X(3,1))=\Aut(Y(3,1),B_Y)$. There is thus an action of $X(3,1)$ onto the set $\{\mathcal{E}_1',\mathcal{E}_3',\mathcal{E}_5',\mathcal{E}_7'\}$. The kernel corresponds to automorphisms of $(T(3,1),B_T)$ acting trivially on the triangle and is thus equal to $\mu_{a,3}$ by Lemma~\ref{Lem:AutoBoundA1}. We obtain  an exact sequence
$$1\to \mu_{a,3}\to \Aut(X(3,1))\to G\to 1,$$
where $G$ is a subgroup of the diedral group $D_8$.
By Lemma~\ref{Lem:AutoBoundA1}, $\sigma_2$ exchanges $\mathcal{E}_1$ and $\mathcal{E}_5$, so corresponds to the transposition $\mathcal{E}_1'\leftrightarrow \mathcal{E}_5'$.  The map $\sigma_3$ exchanges $\mathcal{E}_1$ and $\mathcal{E}_3$ (see Example~\ref{Exa:Fibsigma3}), so corresponds to the permutation  $\mathcal{E}_1'\leftrightarrow \mathcal{E}_3', \mathcal{E}_5'\leftrightarrow \mathcal{E}_7'$.

The map $\sigma_2\sigma_3$ corresponds thus to the permutation $\mathcal{E}_1'\to \mathcal{E}_3'\to \mathcal{E}_5'\to \mathcal{E}_7'$, so the map $\Aut(X(3,1))$ induces a surjective morphism $\langle \sigma_2,\sigma_3\rangle\to D_8$. The explicit formulas for $\sigma_2,\sigma_3$ imply that it is injective.
\end{proof}

\begin{remark} As in the previous two cases, $X(3,1)$ admits a compactification $X$ isomorphic to a cubic surface
$$x_1x_2x_3-x_0^2x_4-x_3^2-x_0^3 = 0.$$
The boundary consists of the union of the line $\ell:x_0 = x_2 = 0$ and the conic $C:x_0 = x_1x_3-x_2^2 = 0$. The points $[0:1:0:0]$ and $[0:0:1:0]$ are singular. The first point is an ordinary node, the second one is of type $A_2$.

 Let $X'\to X$ be a minimal resolution of singularities. The preimage of the boundary is a pentagon of type $(-1,-2,0,-2,-2)$. After we blow up the point $E_2\cap E_3$, and then blow down the curves $E_1,E_5$, we obtain a del Pezzo surface of degree 4 containing $X(3,1)$ with the boundary $B$ equal to a quadrangle of four $(-1)$-curves. This is our compactification $Y(3,1)$. It is known that a del Pezzo surface of degree 4 contains 16 lines in its anti-canonical embedding in $\bbP^4$. There are 40 quadrangles among them, and the Weyl group of type $W(D_5)$ of order $2^4.5!$ acts transitively on this set with the stabilizer isomorphic to the group $\frakS_3\rtimes D_8$ of order $48$, the normalizer of the subgroup $\frakS_3$ of $\frakS_5$.  Our group of automorphisms of $X(3,1)$ is a subgroup of this groups of index 2.
 \end{remark}

\section{Compactifications of $X(a,b)$ with $a,b\ge 2$, with a square}\label{SecSquare}
Let us now study the general case $X(a,b)$ with $a,b\ge 2$ (other cases were  treated in Sections \ref{SecTriangle} and \ref{abg}). The only $(-1)$-curves of the pair $Z(a,b)$ are then $E_4,E_5,E_1$. Denote by $\eta\colon (Z(a,b),B_Z)\to (S(a,b),B_S)$ the birational morphism of pairs which contracts the curve $E_5$.
 The boundary $B_S$ of $X(a,b)$ in $S(a,b)$ consists of a square $\eta(E_1)+\eta(E_2)+\eta(E_3)+\eta(E_4)={\mathcal{E}}_1'+{\mathcal{E}}_2'+{\mathcal{E}}_3'+{\mathcal{E}}_4'$ of type $(0,-(b-1),-(a-1),0)$.

  
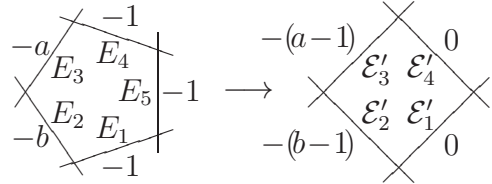
\begin{center}
\xy 
(-75,5)*{};
(-31,-2)*{};(-22,11)*{}**\dir{-};(-29,6)*{-a};(-25,10)*{};(-10,5)*{}**\dir{-};(-17,10)*{-1};(-12,8)*{};(-12,-8)*{}**\dir{-};(-9,0)*{-1};(-10,-5)*{};(-25,-10)*{}**\dir{-};(-17,-10)*{-1};(-22,-11)*{};(-31,2)*{}**\dir{-};(-29,-6)*{-b};
(-24,3)*{E_3};(-18,5)*{E_4};(-15,0)*{E_5};(-18,-5)*{E_1};(-24,-3)*{E_2};
(0,0)*{\longrightarrow};
(8,-2)*{};(22,12)*{}**\dir{-};(8,7)*{-(\!a\!-\!1\!)};(17,3)*{\mathcal{E}_3'};(18,12)*{};(32,-2)*{}**\dir{-};(27,7)*{0};(23,3)*{\mathcal{E}_4'};(32,2)*{};(18,-12)*{}**\dir{-};(27,-7)*{0};(23,-3)*{\mathcal{E}_1'};(22,-12)*{};(8,2)*{}**\dir{-};(8,-7)*{-(\!b\!-\!1\!)};(17,-3)*{\mathcal{E}_2'};
%(-70,5)*{};(-13,-2)*{};(8,10)*{}**\dir{-};(-4,8)*{0};(-1.75,2)*{\mathcal{E}_5};(5,12)*{};(5,-12)*{}**\dir{-};(8,0)*{0};(3,0)*{\mathcal{E}_3};(8,-10)*{};(-13,2)*{}**\dir{-};(-4,-8)*{0};(-1.5,-2.25)*{\mathcal{E}_1};  
 \endxy
 \end{center}
 
 Since the square is standard  (because $a,b\ge 2$), we can apply Proposition~\ref{Prop:Ngons} to describe the automorphism group of $X(a,b)$. 
The description of $S(a,b)$ is given by Corollary~\ref{Cor:Pentagon}:
 \begin{lemma}\label{Lemma:BlowUpSquare}
 The smooth projective surface $S(a,b)$ is the blow-up $\pi'\colon S(a,b)\to \p^1\times \p^1$ of the $a+b$ points 
$$\{([\xi :1],[0:1])\mid\xi^a+1=0\},\hspace{0.5 cm}\{([0:1],[\xi :1])\mid\xi^b+1=0\}$$
and the boundary $S(a,b)\setminus X(a,b)$ consists of the strict transform of the curves $\pi'(\mathcal{E}_1')=\p^1\times [0:1]$, $\pi'(\mathcal{E}_2')=[0:1]\times \p^1$, $\pi'(\mathcal{E}_3')=\p^1 \times [0:1]$ and $\pi'(\mathcal{E}_4')=[1:0]\times \p^1$.
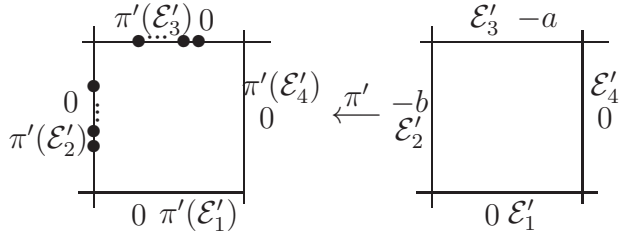
\begin{center}
\xy 
(-50,25)*{};
 (15,-1.5)*{};(15,21.5)*{}**\dir{-};(18,10)*{0};(20,14)*{\pi'(\mathcal{E}_4')};
  (18.5,20)*{};(-6.5,20)*{}**\dir{-};(10,23)*{0};(3,23)*{\pi'(\mathcal{E}_3')};
 (-5,-2)*{};(-5,22)*{}**\dir{-};(-11,7)*{\pi'(\mathcal{E}_2')};(-8,12)*{0};
 (-7,0)*{};(15,0)*{}**\dir{-};(1,-3)*{0};(9,-3)*{\pi'(\mathcal{E}_1')};
 (7,20)*{\bullet};(9,20)*{\bullet};(1,20)*{\bullet};(2.5,20.5)*{\cdot};(3.5,20.5)*{\cdot};(4.5,20.5)*{\cdot}; 
 (-5,8)*{\bullet};(-5,6)*{\bullet};(-5,14)*{\bullet};(-4.5,11.5)*{\cdot};(-4.5,10.5)*{\cdot};(-4.5,9.5)*{\cdot};
 (30,10)*{\longleftarrow}; (30,13)*{\pi'};
 (60,-1.5)*{};(60,21.5)*{}**\dir{-};(63,10)*{0};(63,14)*{\mathcal{E}_4'};
  (63.5,20)*{};(38.5,20)*{}**\dir{-};(54,23)*{-a};(47,23)*{\mathcal{E}_3'};
 (40,-2)*{};(40,22)*{}**\dir{-};(37,8)*{{\mathcal{E}}_2'};(37,12)*{-b};
 (38,0)*{};(62,0)*{}**\dir{-};(48,-3)*{0};(52,-3)*{{\mathcal{E}}_1'};
 \endxy
 \end{center}
Moreover, the restriction of $\pi'$ to the affine surface $X(a,b)$ is given by 
$$(y_1,y_2,y_3,y_4)\mapsto ([y_2:1],[y_3:1]).$$
 \end{lemma}
 \begin{proof}Recall that the birational morphism $\pi\colon Z(a,1)\to \p^2$ of Corollary~\ref{Cor:Pentagon} is the blow-up of the $a+b+2$ points
 $$\{[1:0:\xi]\mid \xi^b+1 = 0\} ,  \{[1:\lambda:0]\mid \lambda^a+1 = 0\}, [0:1:0], [0:0:1].$$
 Composing $\pi$ with the birational map $\kappa\colon \p^2\dasharrow \p^1\times \p^1$  given by $$[x_0:x_1:x_2]\dasharrow ([x_1:x_0],[x_2:x_0]),$$ which blows-up $[0:1:0], [0:0:1]$ and contracts $\pi(E_5)$, we obtain $\pi'$, which is the blow-up of
 $$\{\kappa([1:0:\xi])\mid \xi^b+1 = 0\} ,  \{\kappa([1:\lambda:0])\mid \lambda^a+1 = 0\}.$$
The restriction of  $\pi\colon Z(a,b)\to \p^2$ being $(y_1,y_2,y_3,y_4)\mapsto (1:y_2:y_3)$ (see the proof of Proposition~\ref{Prop:TriangleBarX}), the restriction of $\pi'$ to $X(a,b)$ is 
 $$(y_1,y_2,y_3,y_4)\mapsto \kappa([1:y_2:y_3])=([y_2:1],[y_3:1]).$$\end{proof}

 \begin{lemma}\label{Lem:AutoBoundAB}
The action of the group $\Aut(S(a,b),B_S)$ of automorphisms of the pair $(S(a,b),B_S)$ on the set $\{\mathcal{E}_1',\mathcal{E}_2',\mathcal{E}_3',\mathcal{E}_4'\}$ gives a split exact sequence 
$$1\longrightarrow \mu_{a,b} \to \Aut(S(a,b),B_S)\to H_{a,b}\to 1,$$
where $\mu_{a,b}\simeq \{(\mu,\nu) \in (\k^*)^2\mid \mu^a=\nu^b=1\}$, and $H_{a,b}$ is trivial if $a\not=b$ and isomorphic to $\mathbb{Z}/2\mathbb{Z}$ if $a=b$.

 The group $\mu_{a,b}$ acts on $X(a,b)$ via
 $$(y_1,y_2,y_3,y_4)\mapsto (\nu^{-1}y_1,\mu y_2,\nu y_3,\mu ^{-1}y_4).$$
 
 The group $H_{a,a}$ corresponds to the subgroup of $\Aut(X(a,a))$ generated by 
$$(y_1,y_2,y_3,y_4)\mapsto (y_4,y_3,y_2,y_1).$$
 \end{lemma}
 \begin{proof}
 Denote by $\mu_{a,b}$ the kernel of the action of $\Aut(S(a,b),B_S)$ on the set $\{\mathcal{E}_1',\mathcal{E}_2',\mathcal{E}_3',\mathcal{E}_4'\}$. Let us observe that the set of $a$ curves contracted by $\pi'$ and touching $\mathcal{E}_3$ is invariant by $\mu_{a,b}$. Indeed, the image by $\mu_{a,b}$ of one of the curves is an irreducible curve, not intersecting $\mathcal{E}_1'$ and $\mathcal{E}_2'$. The image of this curve by $\pi'$ does not intersect the two fibres $\pi'(\mathcal{E}_1')=\p^1\times [0:1]$, $\pi'(\mathcal{E}_2')=[0:1]\times\p^1$, so is a point. The same argument works for the $b$ curves contracted by $\pi'$ and touching $\mathcal{E}_2'$.
 
 The group $\mu_{a,b}$ is then the lift of automorphisms of $\p^1\times\p^1$ which leave invariant the four curves $\pi'(\mathcal{E}_i')$, $i=1,\dots,4$ and which preserve the sets 
$\{([\xi :1],[0:1])\mid\xi^a+1=0\},\{([0:1],[\xi :1])\mid\xi^b+1=0\}.$ This group is isomorphic to $ \{(\mu,\nu) \in (\k^*)^2\mid \mu^a=\nu^b=1\}$, acts on  $\p^1\times \p^1$ via
$$([u_1:u_2],[v_1:v_2])\mapsto ([\mu u_1:u_2],[\nu v_1:v_2])$$
 
 and then on $X(a,b)=\Spec \k[y_1,y_2,y_3,y_4]/(y_1y_3-y_2^a\ -1,y_2y_4-y_3^b-1)$ via $$(y_1,y_2,y_3,y_4)\mapsto (\nu^{-1}y_1,\mu y_2,\nu y_3,\mu^{-1} y_4).$$
 
 If $a\not=b$, the action on the set of four curves of $B_S$ is trivial, because the self-intersections have to be preserved. 
 
 If $a=b$, the explicit description of $\pi'\colon S(a,1)\to \p^1\times \p^1$ given in Lemma~\ref{Lemma:BlowUpSquare} shows that the automorphism 
$$([u_1:u_2],[v_1:v_2])\mapsto ([v_1:v_2],[u_1:u_2])$$
of $\p^1\times \p^1$ lifts to an automorphism of $T(a,1)$ which preserves the boundary, exchanging the two $(0)$-curves, and the two $(-a)$-curves.
In affine coordinates, this gives the following automorphism of $X(a,1)\subset \A^4$
$$(y_1,y_2,y_3,y_4)\mapsto \left(y_4,y_3,y_2,y_1\right).$$ \end{proof}

\begin{example}\label{Exa:Fibsigma2ab}
The following birational involution of $\p^1\times\p^1$
$$f\colon ([u_1:u_2],[v_1:v_2])\dashmapsto ([u_1:u_2],[(u_1^a+u_2^a)v_2:v_1u_2^a])$$
is not defined only at $([1:0],[1:0])$ and at $\{([\xi:1],[0:1])\mid \xi^a+1=0\}.$ On the open subset $U\subset \p^1\times\p^1$ where $u_2=1$, we obtain the birational map 
$$([x:1], [v_1:v_2])\dasharrow ([x:1], [(x^a+1)v_2:v_1])$$
whose base-points are $\{([\xi:1],[0:1])\mid \xi^a+1=0\}.$ Hence, the blow-up $\hat{U}\to U$ of these points conjugates $f$ to an automorphism of $\hat{U}$. Since $f$ preserves the set $\{([0:1],[\xi :1])\mid\xi^b+1=0\}$, which is the set of remaining points blown-up by $\pi'$, the map $\pi'$ conjugates $f$ to a birational map $\hat{f}=(\pi')^{-1}f\pi'$ of $S(a,b)$ which restricts to an automorphism  of $S(a,b)\setminus {\mathcal{E}}_4'=(\pi')^{-1}(U)$, and which exchanges $\mathcal{E}_1'$ and $\mathcal{E}_3'$.  

 Since $a\not=0$, the map $\hat{f}$ is not an isomorphism, and is thus a fibered modification $(S(a,b),B_S)\dasharrow (S(a,b),B_S)$. Moreover, $\hat{f}$ restricts to an automorphism of $X(a,b)$, that we will show to be equal to $\sigma_2$.

To compute this, we use the map $X(a,b)\to \p^1\times \p^1$ given by $(y_1,y_2,y_3,y_4)\mapsto  ([y_2:1],[y_3:1])$. The composition with $f$ yields
$$
([y_2:1],[y_3:1])\dashmapsto ([y_2:1],[(y_2^a+1):y_3])=([y_2:1],[y_1:1]).$$
Hence, $y_3$ is exchanged with $y_1$. The involutive automorphism of $X(a,1)$ is thus given by 
$$\sigma_2\colon (y_1,y_2,y_3,y_4)\mapsto (y_3,y_2,y_1,y_1^by_4-y_2^{a-1}\sum_{i=0}^{b-1} (y_1y_3)^i).$$

Similarly, the birational involution of $\p^1\times\p^1$
$$([u_1:u_2],[v_1:v_2])\dashmapsto ([(v_1^b+v_2^b)u_2:u_1v_2^b],[v_1:v_2])$$
yields a fibered modification $(S(a,b),B_S)\dasharrow (S(a,b),B_S)$ which restricts to an automorphism of $S(a,b)\setminus {\mathcal{E}}_1'$ and to the automorphism $\sigma_3$ of $X(a,b)$.

In particular, if  $\psi \colon (S(a,b),B_S)\dasharrow (Y',B')$ is a fibered modification, there is an isomorphism $\tau \colon (Y',B')\to (S(a,b),B_S)$ such that $\tau\psi$ restricts to $\sigma_2$ or $\sigma_3$ on $X(a,b)=S(a,b)\setminus B_S$.
\end{example}

\begin{proposition}\label{GenCase}
If $a,b\ge 2$, then $\Aut(X(a,b))=(\mu_{a,b}\rtimes \langle \sigma_2,\sigma_3\rangle)\rtimes H_{a,b}$, where $\mu_{a,b}$ and $H_{a,b}$ are as in Lemma~$\ref{Lem:AutoBoundAB}$.

Moreover, $\langle \sigma_2,\sigma_3\rangle= \langle \sigma_2\rangle \star \langle \sigma_3\rangle= \mathbb{Z}/2\mathbb{Z} \star \mathbb{Z}/2\mathbb{Z}=\mathbb{Z}\rtimes \mathbb{Z}/2\mathbb{Z}$ is an infinite diedral group, $\mu_{a,b}$ is a finite abelian group, $H_{a,b}$ is trivial if $a\not=b$ and or order $2$ if $a=b$.
\end{proposition}
\begin{proof}
Because $a,b\ge 2$, the pair $(S(a,b),B_S)$ is standard. According to Proposition~\ref{Prop:Ngons}, every automorphism of $X(a,b)$ decomposes into fibered modification and isomorphisms of pairs. Each fibered modification is equal to $\sigma_2$ or $\sigma_3$, up to isomorphism of pairs (Example~\ref{Exa:Fibsigma2ab}), and each automorphism of the pair $(S(a,b),B_S)$ is generated by $\mu_{a,b}$ and $H_{a,b}$ (Lemma~\ref{Lem:AutoBoundAB}). Hence, $\Aut(X(a,b))$ is generated by $\mu_{a,b}$, $H_{a,b}$, $\sigma_2$ and $\sigma_3$. 

In the case where $H_{a,b}$ is not trivial, i.e.\ when $a=b$, we observe that the involution normalises $\mu_{a,b}$ (sending $(\mu,\nu)$ onto $(\nu,\mu)$) and also 
$\langle \sigma_2,\sigma_3\rangle$ (exchanging $\sigma_2$ and $\sigma_3$). To achieve the proof, it remains to observe that $\sigma_2\sigma_3$ is of infinite order. This is of course follows from characterizations of cluster algebras $\calC(a,b)$ with finitely many clusters: they must be of types $A_2,B_2$, or $G_2$ \cite{Sherman}. However, we can give an independent proof. It is  exactly the same as the proof of Proposition~\ref{Prop:Xa14}: the map $\sigma_2\sigma_3$ and its inverse have both a unique proper indeterminacy point, and these two points are different. Proceeding by induction, we obtain that $(\sigma_2\sigma_3)^n$ has again a unique proper indeterminacy point for any $n\ge 1$, always being the proper indeterminacy point of $\sigma_2\circ\sigma_3$.
\end{proof}
\section{Isomorphisms between two surfaces}
We finish this note with the following result.
\begin{proposition}
Let $a,b,c,d\ge 1$. The surfaces $X(a,b)$ and $X(c,d)$ are isomorphic if and only if $(a,b)=(c,d)$ or $(a,b)=(d,c)$.
\end{proposition}
\begin{proof}
If  $(a,b)=(d,c)$, the isomorphism is given by $(y_1,y_2,y_3,y_4)\mapsto (y_4,y_3,y_2,y_1)$. 

Suppose now that $X(a,b)$ is isomorphic to $X(c,d)$. The automorphism groups of $X(a,b)$ and $Y(c,d)$ being isomorphic, the only cases to consider are when $ab\ge 4$ and $cd\ge 4$ (by Theorem~\ref{TheThm}). 

We take a compactification $(Y_1,B_1)$ of $X(a,b)$ by a standard square of type $(0,0,-a,-b)$ (see Section~\ref{SecSquare} and in particular Lemma~\ref{Lemma:BlowUpSquare}), and a compactification $(Y_2,B_2)$ of $X(c,d)$ by a standard square of type $(0,0,-c,-d)$. The isomorphism $X(a,b)\to X(c,d)$ decomposes into fibered modification and isomorphisms of pairs (Proposition~\ref{Prop:Ngons}). These maps do not affect the type of the boundary (which is defined up to permutations), so $(a,b)=(c,d)$ or $(a,b)=(d,c)$.
\end{proof}

\bibliographystyle{plain}

\begin{thebibliography}{30}
\bibitem{Assem} I. Assem, R. Schiffler, V.  Shramchenko, \textit{Cluster automorphisms}. Proc. Lond. Math. Soc. (3) {\bf 104} (2012),1271Ð-1302.
\bibitem{BFZ} A. Berenstein, S.  Fomin, A. Zelevinsky, 
\textit{Cluster algebras. III. Upper bounds and double Bruhat cells}. 
Duke Math. J. {\bf 126} (2005), 1Ð-52.
\bibitem{CAG} I. Dolgachev, \textit{Classical algebraic geometry: a modern view}, Cambridge Univ. Press, 2012.

\bibitem{FZ} S. Fomin,  A. Zelevinsky, \textit{Cluster algebras. I. Foundations}. J. Amer. Math. Soc. {\bf 15} (2002),  
497--529 
\bibitem{Sherman} P. Sherman, A. Zelevinsky, 
\textit{Positivity and canonical bases in rank 2 cluster algebras of finite and affine types}. 
Mosc. Math. J. {\bf 4} (2004),  947--974



\end{thebibliography}

\end{document}